\documentclass{amsart}

\input{epsf.tex}
\usepackage{amsfonts,amssymb,verbatim,amsmath,amsthm,latexsym,textcomp,amscd,hyperref,comment}
\usepackage{latexsym,amsfonts,amssymb,epsfig,verbatim}
\usepackage{amsmath,amsthm,amssymb,latexsym,graphics,textcomp}
\usepackage{graphicx}
\usepackage{color}
\usepackage{url}
\usepackage{stmaryrd}
\usepackage{tikz}
\usepackage{tikz-cd}
\usepackage{mathtools}
\usetikzlibrary{positioning}

\begin{document}

\newtheorem{theorem}{Theorem}[section]
\newtheorem{prop}[theorem]{Proposition}
\newtheorem{lemma}[theorem]{Lemma}
\newtheorem{cor}[theorem]{Corollary}
\newtheorem{prob}[theorem]{Problem}
\newtheorem{defn}[theorem]{Definition}
\newtheorem{notation}[theorem]{Notation}
\newtheorem{fact}[theorem]{Fact}
\newtheorem{conj}[theorem]{Conjecture}
\newtheorem{claim}[theorem]{Claim}
\newtheorem{example}[theorem]{Example}
\newtheorem{rem}[theorem]{Remark}
\newtheorem{assumption}[theorem]{Assumption}
\newtheorem{scholium}[theorem]{Scholium}

\newcommand{\map}{\rightarrow}
\newcommand{\C}{\mathcal C}
\newcommand\AAA{{\mathcal A}}
\def\AA{\mathcal A}

\def\L{{\mathcal L}}
\def\al{\alpha}
\def\A{{\mathcal A}}

\newcommand\GB{{\mathbb G}}
\newcommand\BB{{\mathcal B}}
\newcommand\DD{{\mathcal D}}
\newcommand\EE{{\mathcal E}}
\newcommand\FF{{\mathcal F}}
\newcommand\GG{{\mathcal G}}
\newcommand\HH{{\mathbb H}}
\newcommand\II{{\mathcal I}}
\newcommand\JJ{{\mathcal J}}
\newcommand\KK{{\mathcal K}}
\newcommand\LL{{\mathcal L}}
\newcommand\MM{{\mathcal M}}
\newcommand\NN{{\mathbb N}}
\newcommand\OO{{\mathcal O}}
\newcommand\PP{{\mathcal P}}
\newcommand\QQ{{\mathbb Q}}
\newcommand\RR{{\mathbb R}}
\newcommand\SSS{{\mathcal S}}
\newcommand\TT{{\mathcal T}}
\newcommand\UU{{\mathcal U}}
\newcommand\VV{{\mathcal V}}
\newcommand\WW{{\mathcal W}}
\newcommand\XX{{\mathcal X}}
\newcommand\YY{{\mathcal Y}}
\newcommand\ZZ{{\mathcal Z}}
\newcommand\hhat{\widehat}
\newcommand\flaring{{Corollary \ref{cor:super-weak flaring} }}
\def\Ga{\Gamma}
\def\Z{\mathbb Z}

\def\diam{\operatorname{diam}}
\def\dist{\operatorname{dist}}
\def\hull{\operatorname{Hull}}
\def\id{\operatorname{id}}
\def\Im{\operatorname{Im}}

\def\barycenter{\operatorname{center}}

\def\length{\operatorname{length}}
\newcommand\RED{\textcolor{red}}
\newcommand\BLUE{\textcolor{blue}}
\newcommand\GREEN{\textcolor{green}}
\def\mini{\scriptsize}

\def\acts{\curvearrowright}
\def\embed{\hookrightarrow}

\def\ga{\gamma}
\newcommand\la{\lambda}
\newcommand\eps{\epsilon}
\def\geo{\partial_{\infty}}
\def\bhb{\bigskip\hrule\bigskip}

\title{Propagating quasiconvexity from fibers}

\author{Mahan Mj}
\address{School of Mathematics, Tata Institute of Fundamental Research, 1 Homi Bhabha Road, Mumbai 400005, India}

\email{mahan@math.tifr.res.in}
\email{mahan.mj@gmail.com}
\urladdr{http://www.math.tifr.res.in/~mahan}

\author{Pranab Sardar}
\address{Department of Mathematical Sciences,
	Indian Institute of Science Education and Research Mohali,
	Knowledge City, Sector 81, SAS Nagar,
	Punjab 140306,  India}

\email{psardar@iisermohali.ac.in}
\urladdr{https://sites.google.com/site/psardarmath/}

\thanks{MM is   supported by  the Department of Atomic Energy, Government of India, under project no.12-R\&D-TFR-14001.
	MM is also supported in part by a Department of Science and Technology JC Bose Fellowship, CEFIPRA  project No. 5801-1, a SERB grant MTR/2017/000513, and an endowment of the Infosys Foundation via the Chandrasekharan-Infosys Virtual Centre for Random Geometry.
	This material is based upon work supported by the National Science Foundation
	under Grant No. DMS-1928930 while MM participated in a program hosted
	by the Mathematical Sciences Research Institute in Berkeley, California, during the
	Fall 2020 semester. PS was partially supported by DST INSPIRE grant DST/INSPIRE/04/2014/002236 and DST MATRICS grant 
MTR/2017/000485 of the Govt of India. }

\subjclass[2010]{20F65, 20F67 (Primary), 30F60(Secondary) }

\keywords{quasiconvex subgroup, Cannon-Thurston map, lamination, convex cocompact subgroup, metric bundle, mapping class group, $Out(F_n)$}

\date{\today}

\begin{abstract}  Let
$1 \to K \longrightarrow G \stackrel{\pi}\longrightarrow Q$
be an exact sequence of hyperbolic groups. Let $Q_1 < Q$ be a quasiconvex
subgroup and let $G_1=\pi^{-1}(Q_1)$. Under  relatively mild conditions (e.g.\ if $K$ is a closed surface group or a free group and $Q$ is convex cocompact), we show that infinite index quasiconvex subgroups of $G_1$ are quasiconvex   in $G$. Related results are proven for metric bundles, developable complexes of groups and graphs of groups.
\end{abstract}
\maketitle


\section{Introduction} The aim of this paper is to provide evidence in favor of the following Scholium.
\begin{scholium}\label{schol-propagate}
For an exact sequence $$1 \to K \longrightarrow G \stackrel{\pi}\longrightarrow Q\to 1$$  of hyperbolic groups, and more generally for hyperbolic metric bundles, quasiconvexity in fibers propagates to quasiconvexity in subbundles.
\end{scholium}
There are two specific examples in which the exact sequence in Scholium \ref{schol-propagate} has been studied:

\begin{enumerate}
\item When $K=\pi_1(\Sigma)$ is the fundamental group of a closed surface of genus $g>1$, and $Q<MCG(\Sigma)$. Here, the 
exact sequence comes from the Birman exact sequence
$$1 \to \pi_1(\Sigma) \longrightarrow MCG(\Sigma.\ast) \stackrel{\pi}\longrightarrow MCG(\Sigma)\to 1,$$ where $MCG(\Sigma.\ast)$ denotes the mapping class group of 
$\Sigma$ equipped with a marked point $\ast$.
\item When $K=F_n$ is the free group on $n$ generators ($n>2$), and $Q<Out(F_n)$. Here, the 
exact sequence comes from the Birman exact sequence for free groups
$$1 \to F_n \longrightarrow Aut(F_n) \stackrel{\pi}\longrightarrow Out(F_n)\to 1.$$
\end{enumerate}
 
The results in this paper are in the same vein as
 a series of theorems starting with work of Scott-Swarup \cite{scottswar}, followed by several authors \cite{mitra-pams,dkl,kl15,dkt,dt1,mahan-rafi,ghosh}. In all these papers, the setup was as follows: Consider
  an exact sequence of hyperbolic groups as in Scholium \ref{schol-propagate}. Then, under relatively mild conditions, 
 it was shown that an infinite index quasiconvex subgroup of $K$ is quasiconvex in the bigger group $G$.
 The purpose of this paper is to extend these results in a different direction. Let $Q_1 < Q$ be a quasiconvex
 subgroup and let $G_1=\pi^{-1}(Q_1)$. The main results of this paper show similarly that,
 under  some conditions, infinite index quasiconvex subgroups of $G_1$ are quasiconvex   in $G$. In other words the distortion \cite[Chapter 4]{gromov-ai} (see Definition \ref{defn-disto} below) is entirely captured by the fiber group. The earlier papers cited above all treat the case with $Q_1=\{1\}$. 
 
The simplest example that illustrates this
 is the following. Let $M_1, M_2$ be two
closed hyperbolic 3-manifolds fibering over the circle with the same topological fiber $F$. Suppose further that the 3-complex $M_1 \cup_F M_2 =M$ obtained by gluing $M_1, M_2$ along the fiber $F$ has
a hyperbolic fundamental group given by $G=\pi_1(M_1)\ast_{\pi_1(F)} \pi_1(M_2)$.
Then a prototypical theorem of this paper  shows that 
an infinite index quasiconvex subgroup of $\pi_1(M_1)$ is quasiconvex in $G$.
Analogous results are also shown for $K$ a free group.
Note that, using work of \cite{farb-coco,kl-coco,hamen,dt1}, these results can equivalently be formulated in terms of   quasiconvex
subgroups $Q_1$ of convex cocompact subgroups of the mapping class group $MCG(\Sigma)$
or the outer automorphism group $Out(F_n)$ of the free group as follows
(see Theorems \ref{surface fiber} and \ref{freefiber}):

\begin{theorem}\label{thm-intro}
 Let $$1 \to K \longrightarrow G \stackrel{\pi}\longrightarrow Q\to 1$$ 
 be an exact sequence of hyperbolic groups, 
 where $K$ is either  $\pi_1(\Sigma)$, 
with $\Sigma$ a closed surface of genus at least $2$; or a finitely generated free group $F_n$, $n>2$. Let $Q$ be   an infinite convex cocompact subgroup
of respectively $MCG(\Sigma)$ or $Out(F_n)$. 
Let $Q_1< Q$ be a qi embedded subgroup and $G_1=\pi^{-1}(Q_1)$. 
	Suppose $H<G_1$ is an infinite index, quasiconvex subgroup. Then $H$ is quasiconvex  in $G$.
\end{theorem}

The main new technical tool used in the proof of Theorem
\ref{thm-intro} is the existence and structure
of the Cannon-Thurston map for the pair $(G_1,G)$ established in \cite{mbdl2}.
We show also (Proposition \ref{prop-ctreg}) that in the absence of convex cocompactness, Theorem \ref{thm-intro} fails quite dramatically due to the existence of subgroups $K_1$ of $K$ that are quasiconvex in $K$ and $G_1$,
but not in $G$. The main theorems of the paper are more
general, and make the content of Scholium \ref{schol-propagate} precise in three contexts: 
\begin{enumerate}
\item exact sequences of groups (Theorem \ref{general fiber}),
\item complexes of groups  (Theorem \ref{thm-cx}), and 
\item general graphs of groups
(Theorem \ref{thm-graph}).
\end{enumerate}

The underlying geometric structure in all three cases is given by
a map $\pi:\XX \to \BB$ from a total space $\XX$ to a base space $\BB$. In the first two cases, $\pi:\XX \to \BB$ is a metric graph bundle (see Definition \ref{defn-mgbdl}), which may be thought of as a natural analog of a bundle in the context of coarse geometry. In the third case,  $\pi:\XX \to \BB$ is a tree of spaces (cf.\ \cite{BF}). We assume that both $\XX, \BB$ are hyperbolic
and consider a qi embedded subspace $\BB_1 \subset \BB$. The pullback
$\XX_1 =  \pi^{-1} (\BB_1)$ gives a subbundle or a subtree of spaces. In the cases we are interested in, $\XX$ corresponds to a group $G$, and $\XX_1$
to a subgroup $G_1$ of $G$ (a standard motif of geometric group theory via
the Milnor-{\v S}varc lemma
\cite[Proposition 8.19]{bridson-haefliger} for instance).
Then,
one looks at a subgroup $H<G_1$ such that
\begin{enumerate}
\item $H$ is qi embedded in $G_1$, 
\item $H$ is not qi embedded in $G$, or equivalently, $H$ is distorted in $G$
(see Definition \ref{defn-disto}).
\end{enumerate}
The main aim of all three Theorems \ref{general fiber}, 
\ref{thm-cx}, and \ref{thm-graph} is to identify the source of distortion
as coming from a fiber $\FF_v = \pi^{-1} (v)$ of $\pi:\XX \to \BB$. Again,
in all three cases, the fiber $\FF_v $ corresponds to a subgroup $K<G$ (in Scholium \ref{schol-propagate}, it is the normal subgroup). The output of 
Theorems \ref{general fiber}, 
\ref{thm-cx}, and \ref{thm-graph} is a finitely generated subgroup of 
$H \cap K$ that is distorted in $G$. Thus, the source of distortion of $H$ in $G$ is the existence of a finitely generated distorted subgroup in the intersection
of $H$ with the
 fiber group.

\section{A coarse topological fact}\label{sec-coarsetop} The main aim of this section is to prove Proposition \ref{general thm}, which is a  generalization in the geometric group theory
context of the following simple 3-manifold fact: 
\begin{fact}\label{fact}{\rm 
Let $M$ be a closed hyperbolic 3-manifold fibering over the circle with fiber $F$. Let $\Sigma \subset M$
be an immersed incompressible quasi-Fuchsian 
surface. Let $G=\pi_1(M), K=\pi_1(F), H=\pi(\Sigma)$. Let $M_F$ denote the
cover of $M$ corresponding to $K<G$, let $F_0$ denote a lift of $F$ to $M_F$ and $\Sigma_F$ denote a lift of $\Sigma$ to $M_F$. 
Then, given any  finitely generated subgroup $K_1 < K \cap H$, there is a compact subsurface $\Sigma_0
\subset \Sigma_F$, necessarily contained in a finite neighborhood of $F_0$, such that $K_1 < i_\ast (\pi_1(\Sigma_0))$, where $i: \Sigma_0 \to M_F$ denotes the inclusion map.}
\end{fact}

While we could not find an explicit reference for Fact \ref{fact}, the proof is fairly straightforward. Let $\Sigma_F^0$ denote the cover
of $\Sigma$ corresponding to $K \cap H$. Then 
 $\Sigma_F \subset M_F$ comes from an immersion $\iota$ of $\Sigma_F^0$ into 
 $M_F$ such that $$\iota_\ast: \pi_1(\Sigma_F^0) \to \pi_1(M_F) = \pi_1(F)$$ is injective. Then $\iota_\ast( \pi_1(\Sigma_F^0)) =
 K \cap H$ is infinitely generated free. Hence $K_1$ can be
 identified with a finitely generated free subgroup of 
$ \pi_1(\Sigma_F^0)$, and is therefore carried by a \emph{compact}
essential subsurface $W$ of $\Sigma_F^0$. Choosing $\iota (W)=\Sigma_0$,
Fact \ref{fact} follows.

\subsection{Preliminaries}\label{sec-prel}
{\em All  metric spaces  in this paper are proper geodesic metric spaces.} A {\em metric graph} is a 
simplicial connected graph where all the edges are assigned length $1$ and the graph is given the induced length metric \cite[Chapter I.1, Section 1.9]{bridson-haefliger}. 
For any graph $X$ we shall denote by $V(X)$ and $E(X)$ the vertex set and the edge set of $X$ respectively.
For any metric space $X$, $A\subset X$ and $r\geq 0$ we denote by $N_r(A)$ the set $\{x\in X: d(x,a)\leq r\,\mbox{for some}\, a\in A\}$
 and refer to it as the $r$-neighborhood of $A$ in $X$.
Also for $x\in X$, $r\geq 0$ we shall denote by $B(x,r)$ the closed ball of radius $r$ in $X$ centered at $x$.
Given $A,B\subset X$ the Hausdorff distance of $A,B$ is given by $Hd(A,B)=\inf \{r\geq 0: B\subset N_r(A), A\subset N_r(B)\}$.
We now recall some basic notions we shall use in this paper.
We refer the reader to \cite[Chapter I.8]{bridson-haefliger} for details on quasi-isometries and q(uasi)-i(sometric) embeddings,
 and to \cite[Chapter III.H]{bridson-haefliger} for
details on (Gromov-)hyperbolic groups. 
If $X$ is a geodesic metric space and $Y\subset X$ is a path-connected subspace we equip $Y$ with the
induced path metric
\cite[Chapter I.3]{bridson-haefliger} and assume that $Y$ is a geodesic metric space
with respect to this metric.

If $G$ is a group with a finite generating set $S$ then we shall always identify $G$ with the vertex set of the Cayley
graph $\Gamma(G,S)$. The metric on $G$ induced from $\Gamma(G,S)$ will be referred to as the {\em word metric} on $G$ with respect to
the generating set $S$.
Let $H$ be a finitely generated subgroup of a finitely generated group $G$.
Assume further that a finite generating set of $G$ is chosen extending a 
finite generating set of $H$, so that the Cayley graph $\Gamma_H$ embeds in
$\Gamma_G$.

\begin{defn}\label{defn-disto}\cite[Chapter 4]{gromov-ai}
Let $i: \Gamma_H\to \Gamma_G$ be the above embedding of  Cayley graphs. The \emph{distortion function} of $H$ in $G$ is given by 
$$disto(R) = Diam_H (\Gamma_H \cap B_G(R))$$ where $B_G(R)$ is the $R-$ball about $1\in \Gamma_G$. If the distortion is bounded below by a 
super-linear function, we say that $H$ is \emph{distorted} in $G$. Else, we say it is \emph{undistorted}.
\end{defn}

Note that $H$ is undistorted in $G$ if and only if it is qi-embedded.
Given a function $f:\mathbb N\map \mathbb N$, a family of maps $\phi_{i}: X_{i}\map Y_{i}, \, i \in I$
between metric spaces is called {\em uniformly metrically proper 
	as measured by $f$} \cite{mahan-sardar} if for all $i\in I$ and for all $x,x'\in X_{i}$, 
$$d_{Y_{i}}(\phi_{i}(x), \phi_{i}(x'))\leq R \Rightarrow d_{X_{i}}(x,x')\leq f(R)$$
for all $R\geq 0$.
We record the following elementary fact for easy reference.
\begin{lemma}\label{lem: prop emb}
(1) Suppose $H<G$ are finitely generated groups. Then the inclusion $H\map G$ is uniformly metrically proper
with respect to any given word metrics on $G,H$.\\
(2) Suppose $X$ is a geodesic metric space and $Y$ is a subspace such that it is also a geodesic metric space with respect to the induced length
metric $d_Y$. Suppose that the inclusion map $Y\map X$ is uniformly metrically
proper as measured by $f:\NN\map \NN$ with respect to these metrics. Then given a $k$-quasigeodesic $\alpha$ of $X$ contained in $Y$, 
it is a $k'=k'(k,f)$-quasigeodesic of $Y$ with respect to the metric $d_Y$.
\end{lemma}

A  geodesic metric space $X$    is $\delta$-{\em hyperbolic} 
if for any geodesic triangle $\Delta xyz$, $[x,y]\subset N_{\delta}([y,z]\cup [x,z])$.
A metric space $X$ is called {\em hyperbolic} if it is $\delta$-hyperbolic for some $\delta\geq 0$. 
A finitely generated group is hyperbolic if some (any) of its Cayley graphs (with respect to a finite generating set) is hyperbolic. 
Given a hyperbolic metric space $X$ its {\em geodesic boundary or visual boundary} is the set of equivalence classes of geodesic rays where two
rays $\alpha, \beta$ are equivalent if and only if $Hd(\alpha, \beta)<\infty$. The boundary of $X$ is denoted by $\partial X$. 
The set $\bar{X}:=X\cup \partial X$ comes equipped with a natural topology \cite[Chapter III.H]{bridson-haefliger}.
Suppose $X$ is a hyperbolic metric space and $\alpha:[0,\infty)\map X$ is a geodesic ray. 
Then the equivalence class of $\alpha$ is denoted by $\alpha(\infty)$. 
Also if $\alpha:[0,\infty)\map X$ is a quasigeodesic ray, then there is a geodesic ray $\beta:[0,\infty)\map X$, unique up to
the equivalence relation  above, such that $Hd(\alpha,\beta)<\infty$. We shall use the same notation $\alpha(\infty)$ to denote 
$\beta(\infty)$.
If $G$ is a hyperbolic group then we write $\partial G$ to denote the geodesic boundary of any of its
Cayley graphs since any two such boundaries of Cayley graphs are  homeomorphic (cf.\ Lemma \ref{ct elem lemma}(2)).

\begin{comment}

A map $\phi:X\map Y$ is a {\em quasi-isometric embedding or simply
qi embedding} if there is a constant $K\geq 1$ such that for all $x,x'\in X$ we have 
$$-K+\frac{1}{K}d_X(x,x')\leq d_Y(\phi(x), \phi(x'))\leq K+Kd_X(x,x').$$

	content...

\begin{enumerate}
\item Suppose  is a map between metric spaces. We will say that $\phi$ 
\item ()
\item By a geodesic (quasi geodesic) in a metric space $X$ we shall mean an isometric (resp. qi ) embedding of an interval $I\subset \RR$
in $X$. If $I=[0,\infty)$ we shall call it a geodesic (resp. quasi-geodesic) ray. If $I=\RR$ then we shall call it a geodesic (resp. quasi-geodesic)
line in $X$. If $x,y\in X$ we shall denote by $[x,y]$ a geodesic joining $x,y$. For all $x,y,z\in X$, we denote by $\Delta xyz$ any
set of three geodesics $[x,y], [y,z], [x,z]$ and call it a {\em geodesic triangle}. The segments $[x,y], [y,z], [x,z]$ are called the
{\em sides} of the triangle.

\item 
\end{enumerate}
\end{comment}

If $\alpha:\RR\map X$ is a (quasi)geodesic line then by $\alpha(-\infty)$ we shall denote the point of $\partial X$ determined by 
the ray $t\mapsto \alpha(-t)$. In this case, we say that $\alpha$ joins the pair of points $\alpha(-\infty), \alpha(\infty)$. 
Suppose $\xi_i$, $i=1,2,3$
are any three distinct points of $\partial X$. Then any set of three geodesic lines joining these points in pairs is called an 
{\em ideal triangle} with vertices $\xi_1, \xi_2, \xi_3$. The geodesic lines are called the {\em sides} of the ideal triangle.

\begin{lemma}\label{lem-bary}
Suppose $X$ is proper $\delta$-hyperbolic. Then:\\
(1) {\em (Visibility \cite[Chapter III.H, Lemma 3.2]{bridson-haefliger}).}  Suppose $\xi_1, \xi_2\in \partial X$ are two distinct points. Then there is a geodesic line in $X$ joining them.\\
(2) {\em (Barycenters of ideal triangles) \cite[Chapter III.H, Lemmas 1.17, 3.3]{bridson-haefliger}} There are constants $D=D(\delta), R=R(\delta)$ such that the following holds.
For any three distinct points $\xi_1,\xi_2, \xi_3\in \partial X$ there is a point $x\in X$ which is contained in the $D$-neighborhood
of each of the three sides of any ideal triangle with vertices $\xi_1,\xi_2, \xi_3$. Moreover, 
if $x'$ is any other such point then $d(x,x')\leq R$.  
\end{lemma}
A point $x$ as in Lemma \ref{lem-bary} 
is referred to as a {\em barycenter} of the ideal triangle. 
Thus we have a coarsely well-defined map $\partial^3 X\map X$ sending distinct triples of points to a barycenter. We call such a map a {\em barycenter map.}
We call $X$ a {\em non-elementary hyperbolic metric space} if there exists a constant $L\geq 0$ such that the $L$-neighborhood of the image of a barycenter
map is the entire $X$ (the terminology is based on the notion of a non-elementary group of isometries 
of a hyperbolic metric space).

\begin{defn}
Suppose $X$ is a hyperbolic metric space and $A\subset X$.
Then the {\em limit set} of $A$ in $X$, denoted by $\Lambda_X(A)$,
is the collection of accumulation points of $A$ in $\partial X$.
\end{defn}
If $G$ is a hyperbolic group and $H$ is a subgroup then  the limit set
$\Lambda_G(H)$ of $H$ in $G$ is the limit set of some (any) $H-$orbit.
The following is a list of basic properties of the limit set that will be useful for us (see \cite[Chapter 8]{thurstonnotes} and \cite[Chapter III.H.3]{bridson-haefliger} for instance).

\begin{lemma}\label{limit set lemma}
(0) If $X$ is a hyperbolic metric space and $A\subset X$ then $\Lambda_X(A)$ is a 
closed subset of $\partial X$; $\Lambda_X(A)=\emptyset$ if and only if  $A$ is bounded.\\
(1) If $X$ is a hyperbolic metric space and $A,B\subset X$ then $\Lambda_X(A)=\Lambda_X(B)$ if
$Hd(A,B)<\infty$. If $Y $ is qi-embedded in $X$, then $\partial Y$ embeds
in $\partial X$.\\
(2) Suppose $H<G$ are hyperbolic groups. Then the following hold:\\
(i) If $H\trianglelefteq G$ then $\Lambda_G(H)=\partial G$.\\
(ii) If $H$ is a finitely generated subgroup of $G$ which is qi embedded and is of infinite index
then $\Lambda_G(H)$ is a proper closed subset of $\partial G$.
\end{lemma}

We will also need the following elementary lemma. 

\begin{lemma}\label{Hd lemma}
Let $G$ be a finitely generated group acting properly and cocompactly 
by isometries
on a metric space $X$.
Let $\Gamma$ be a Cayley graph of $G$ with respect to a finite generating set.
Suppose $H,K$ are two finitely generated subgroups of $G$ and $A,B$ are two subsets of $X$ invariant under $H$ and $K$ respectively.
Lastly, assume that $A/H, \, B/K$ are compact. If $Hd_X(A,B)<\infty$ then $Hd(H,K)<\infty$ where  $Hd(H,K)$ 
is taken in $\Gamma$.
\end{lemma}

\begin{proof}
 Let $S$ be a generating set of $G$ and let $\Gamma$ be the Cayley graph $\Gamma(G,S)$. On subsets of $G$ we shall
use the metric induced from $\Gamma$.
Let $x\in A, y\in B$. Then $Hd(K.x, K.y)\leq d_X(x,y)<\infty$. Since the actions of $H, K$ on $A,B$ respectively
are cocompact, we have $Hd(H.x,A)<\infty, Hd(K.y, B)<\infty$. Since $Hd(A,B)<\infty$,  it follows that
$Hd(H.x, K.y)<\infty$. Hence $Hd(H.x, K.x)<\infty$. Since the action of $G$ on $X$ is proper and
cocompact, the orbit map $G\map X$ given by $g\mapsto g.x$ for all $g\in G$ is a quasi-isometry by the Milnor-{\v S}varc lemma
(see \cite[Proposition 8.19]{bridson-haefliger}). Therefore,
 $Hd(H,K)<\infty$.
\end{proof}

\subsection{Finitely generated subgroups from bounded neighborhoods}
We are now in a position to generalize Fact \ref{fact}.

\begin{prop}\label{general thm}
Let $G$ be a finitely presented group and $H,K$ be two subgroups where
	$H$ is finitely generated. Let $S\subset G$, $S_1\subset H$ be finite generating sets and $\Gamma=\Gamma(G,S)$,
	$\Gamma_1=\Gamma(H,S_1)$ be the corresponding Cayley graphs.
	Let $D>0$ and suppose that there is an infinite set $A\subset N_D(K)\cap H$ where $N_D(K)$ is the $D$-neighborhood of $K$ in $\Gamma$.
	Suppose that 
    there exists $r\geq 1$ such that any pair of points of $A$ can be connected by a path lying in
	the $r$-neighborhood of $A$ in $\Gamma_1$.	
	Then there is a finitely generated infinite subgroup $K_1<H\cap K$ such that  $A/K_1$ is finite.
In particular, $A$ is contained within a finite neighborhood of $K_1$ in $\Gamma$.
\end{prop}
\begin{proof}
We start with a finite connected simplicial 2-complex $\YY$ together with a finite 1-subcomplex $\ZZ$ such that 
 $G= \pi_1(\YY)$ and $i^*_{\ZZ,\YY}(\pi_1(\ZZ))=H$
where $i_{\ZZ,\YY}:\ZZ\map \YY$ is the inclusion map and $i^*_{\ZZ,\YY}:\pi_1(\ZZ)\map \pi_1(\YY)$ is the induced map. One
may metrize $\YY$ in the standard way so that it is a geodesic metric space (\cite[Chapter I.7, Theorem 7.19]{bridson-haefliger}). 
Fix one such metric on $\YY$ once
and for all. Let $p:\XX\map \YY$ be the universal cover of $\YY$ endowed with the induced length metric from $\YY$.
(See \cite[Definition 3.24, Chapter I.3]{bridson-haefliger}.) We note that $\XX$ is a proper metric space 
(\cite[Exercise 8.4(1), Chapter I.8]{bridson-haefliger}).
Let $x\in \XX$ be a point such that $p(x)\in \ZZ$.  $G$ acts on $\XX$ by deck
transformations so that $\XX/G=\YY$. Let $\phi:G\map \XX$, $g\mapsto g.x$ be the orbit map. 
Then, identifying $G$ with $V(\Gamma)$ (equipped with the subspace metric),
 $\phi$ gives a Lipschitz map from $G$ 
to $\XX$. Similarly, the inclusion map $i_{H,G}:H\map G$ is also Lipschitz. Therefore, the composition $\phi\circ i_{H,G}$
is Lipschitz. Suppose that $\phi$ and $\phi\circ i_{H,G}$ are both $L$-Lipschitz.

Let $\tilde{\ZZ}$ be the connected component of $p^{-1}(\ZZ)$ containing $x$. It follows from the hypothesis that if $R=(D+r)L$ then
$K.B(x;R)$ contains a connected set $\FF\subset \tilde{\ZZ}$ such that $\phi(A)\subset \FF$. Now consider the quotient map $q:\XX\map \XX/K$.
Since $\XX$ is a proper metric space, $B(x,R)$ is compact. Hence, $q(\FF)$ is contained in the compact set $q(B(x;R))$. 
Thus $\overline{q(\FF)}$ is compact. 
We note that $q(\tilde{\ZZ})$ is a closed subcomplex of $\XX/K$. Hence, there is a finite subcomplex, say $\WW$, of $q(\tilde{\ZZ})$ 
containing $\overline{q(F)}$. Let $i: \WW\map q(\tilde{\ZZ})$, and $j: q(\tilde{\ZZ})\map \XX/K$ be the inclusion maps. 
Let $K_1$ be the image of $j^*\circ i^*$.
Clearly this is a finitely generated subgroup of $H\cap K$.
Since $q^{-1}(\WW)$ contains the set $\FF$ it follows that there is a connected component, say $\tilde{\WW}$, of $q^{-1}(\WW)$
containing $\FF$. Since $\tilde{\WW}/K_1$ is compact and $\phi(A)\subset \FF\subset \tilde{\WW}$, it follows that $\phi(A)/K_1$ is finite.

 Finally, we note that the orbit map $\phi: G\map \XX$ is a $G$-equivariant quasi-isometry by the Milnor-{\v S}varc lemma 
(\cite[Proposition 8.19]{bridson-haefliger}) since the $G$-action on
$\XX$ is proper and cocompact. It follows that $A/K_1$ is finite. Since $A$ is an infinite set,
it follows that $K_1$ is also infinite. \end{proof}

\section{Graphs of  spaces and Cannon-Thurston Maps} We shall recall some material from \cite{mahan-sardar,mahan-rafi,mbdl2} and deduce some consequences.
Informally a graph of metric spaces is a $1$-Lipschitz surjective map $\pi:\mathcal X\map \mathcal B$ from $\XX$ to a   metric graph $\BB$ satisfying
some additional  conditions. We refer to $\mathcal X$ as the {\em total space} and $\mathcal B$ as the {\em base}. 
We shall need  two specific instances of this:  metric graph bundles and trees of metric spaces.
The common feature in both these cases is that $\mathcal X$ is a graph, the map $\pi$ is simplicial and the fibers are uniformly properly 
embedded in the total space. As usual, we shall use $V(\GG)$ to denote the vertex set of a graph $\GG$.

\subsection{Metric graph bundles}
\begin{defn}\label{defn-mgbdl}\cite[Definition 1.2]{mahan-sardar}
Suppose $\mathcal X$ and $\BB$ are metric graphs and $f:\mathbb N \rightarrow \mathbb N$ is a function.
We say that $\XX$ is an $f$-{\em metric graph bundle} over $\BB$ if there exists a surjective simplicial 
map $\pi:\XX\rightarrow \BB$  such that the following hold.
\begin{enumerate}
\item For all $b\in V(\BB)$, $\FF_b:=\pi^{-1}(b)$ is a connected subgraph of $\XX$. Moreover, the inclusion maps
$\FF_b\rightarrow \XX$, $b\in V(\BB)$ are uniformly metrically proper as measured by $f$.
\item For all adjacent vertices $b_1,b_2\in V(\BB)$, any $x_1\in V(\FF_{b_1})$ is connected 
by an edge to some $x_2\in V(\FF_{b_2})$.
\end{enumerate}
\end{defn}

For all $b\in V(B)$ we shall refer to $F_b$ as the {\em fiber} over $b$ and denote its path metric by $d_b$.
Condition (2) of Definition \ref{defn-mgbdl} immediately gives:

\begin{lemma}\label{bundle lemma1}
If $\pi:\XX\rightarrow \BB$ is a metric graph bundle then for any points $v,w\in V(\BB)$ we have
$Hd(F_v, F_w)<\infty$.
\end{lemma}

\begin{comment}
	content...

\proof This is clearly true by c. Once can then use induction on the
distance of $d(v,w)$ to prove it in general.
\qed

The following notion will be important for our purpose too.
\end{comment}

\begin{defn}
Suppose $\XX$ is an $f$-{\em metric graph bundle} over $\BB$.
Given $k\geq 1$ and a connected subgraph $\AAA\subset \BB$, a {\em $k$-qi section} over $\AAA$ is
a map $s:\AAA\map \XX$ such that $s$ is a $k$-qi embedding and $\pi\circ s$ is the identity
map on $\AAA$.
\end{defn}
We shall only  need qi sections over geodesics  in this paper.
We next discuss the main examples of metric graph bundles that we shall refer to later. 

\begin{lemma}\label{lemma: restriction bundle}{\em (Restriction of metric graph bundles \cite[Lemma 3.17]{mbdl2})}
Suppose $\pi:\XX\map \BB$ is an $f$-metric graph bundle and $\BB_1$ is a connected subgraph of $\BB$. Let $\XX_1=\pi^{-1}(\BB_1)$.
Then clearly $\XX_1$ is a connected subgraph of $\XX$. Let    $\pi_1:\XX_1\map \BB_1$
denote the restriction of 
$\pi$ to $\XX_1$. Then $\pi_1:\XX_1\map \BB_1$ is also an $f$-metric graph bundle.
\end{lemma}
\begin{example}\label{ex: exact sequence}{\em (Metric graph bundles from short exact sequences \cite[Example 1.8]{mahan-sardar}, \cite[Example 5]{mbdl2}.)} 
{\em Suppose we have a short exact sequence of finitely generated groups 
$$1\rightarrow K\stackrel{i}{\rightarrow} G\stackrel{\pi}{\rightarrow} Q\rightarrow 1.$$
Suppose $S$ is a finite generating set of $G$ such that $S$ contains a generating set $A$ of $K$. 
Let $\XX=\Gamma(G,S)$ be the Cayley graph of $G$ with respect to the generating set $S$.
Let $B=\pi(S) \setminus \{1\}$ and $\BB:=\Gamma(Q,B)$ be the Cayley graph of the group $Q$ with respect to the generating set $B$.
Then  the map $\pi$ naturally induces a simplicial map  $\pi:\XX\rightarrow \BB$ between  Cayley graphs.
This is a metric graph bundle. The fibers are the translates of
$\Gamma(K,A)$ under left multiplication by elements of $G$.

Moreover, suppose $Q_1<Q$ is a finitely generated subgroup and $G_1=\pi^{-1}(Q_1)$. Suppose $B$ contains a generating set $B_1$ of $Q_1$.
Let $S_1=S\cap G_1$, $\XX_1=\Gamma(G_1, S_1)$, and $\BB_1=\Gamma(Q_1, B_1)$. Then $\pi$ restricts to a metric graph bundle map
$\pi_1:\XX_1\map \BB_1$ by Lemma \ref{lemma: restriction bundle}.}
\end{example}


 \noindent {\bf Metric graph bundles  from complexes of groups \cite[Example 3]{mbdl2}:}\\
We refer to \cite{bridson-haefliger} and \cite{haefliger-cplx} for basics on developable complexes of (finitely generated) groups.

\begin{defn}
Suppose $\mathcal Y$ is a finite connected simplicial complex and $\GB(\mathcal Y)$ is a developable complex of finitely generated groups
over $\YY$. Let $G=\pi_1(\GB(\YY))$ be the fundamental group of the complex of groups. (See \cite[Definitions 3.1, 3.5, Chapter III.$\mathcal C$]{bridson-haefliger}).

We shall call $\GB(\mathcal Y)$  a {\em complex of groups with qi condition} if
for all faces $\tau \subset \sigma$ of $\mathcal Y$ the corresponding homomorphism $G_{\sigma}\map G_{\tau}$
is an isomorphism onto a finite index subgroup of $G_{\tau}$.

If moreover all face groups are (nonelementary) hyperbolic then we shall refer to $\GB(\mathcal Y)$ as a
 {\em complex of (nonelementary) hyperbolic groups with qi condition}. 
\end{defn}

\begin{prop}\cite[section 3.2]{mbdl2}\label{prop1: cplx gps}
 Suppose $\mathcal Y$ is a finite connected simplicial complex and $\GB(\mathcal Y)$ is a developable complex of finitely generated groups
with qi condition.

Then there is a metric graph bundle $\pi:\XX\map \BB$ where $G$ acts on both $\XX$ and $\BB$ by isometries such that the following hold.\\
(1) The map $\pi$ is $G$-equivariant.
\\
(2) The $G$-action is proper and cocompact on $\XX$.  The $G$-action is  cocompact (but not necessarily proper) on $\BB$. 
\\
(3) There is an isomorphism of graphs $p: \BB/G\map \YY^{(1)}$ such that for all $\sigma_0\in \YY^{(0)}$, 
 and $v\in p^{-1}(\sigma_0)$, $G_v$ is a conjugate of $G_{\sigma_0}$ in $G$.
\\
(4) For all $v\in V(\BB)$, the $G_v$-action on $V(\FF_v)$ is transitive but the action on $E(\FF_v)$ has a uniformly bounded number of orbits. Thus, the
$G_v$-action on $\FF_v$ is proper and cocompact. 
In particular if all the groups $G_{\sigma}$ are hyperbolic then the fibers of the metric graph bundle $\pi:\XX\map \BB$ are uniformly
hyperbolic. (Note, in particular, that since $\YY$ is finite, the collection of $G_{\sigma}$'s is uniformly
hyperbolic.)\\
(5) Suppose moreover that $\YY_1\subset \YY$ is a connected subcomplex such that the inclusion morphism $\GB(\YY_1)\map \GB(\YY)$
induces an injective homomorphism at the level of fundamental groups (see \cite[Proposition 3.6, Chapter III.$\mathcal C$]{bridson-haefliger}). 
Let $G_1$ be the image of $\pi_1(\GB(\YY_1))$ in $G$. Then we can construct a metric graph bundle $\pi:\XX\map \BB$ as above along with a
connected subgraph $\BB_1\subset \BB$ invariant under $G_1$ such that $p(\BB_1/G_1)=\YY_1$, and $G_1$ acts properly and
cocompactly on $\XX_1=\pi^{-1}(\BB_1)$.
\end{prop}

\begin{rem}{\em
We note that $\BB$ in the above proposition is quasi-isometric to the metric graph obtained from the Cayley graph of $G$
after {\it coning off} (see \cite[Section 3.1]{farb-relhyp}) the various cosets of the face subgroups of $G$. A similar remark applies to $\BB_1$
in  part (5) of the above proposition. Later in the paper we shall also assume that $\BB_1$
is qi embedded in $\BB$. Since this hypothesis is used repeatedly we abstract this out
as a definition for  convenience of  exposition.

}
\end{rem}

\begin{defn}\label{defn: good subcomplex}
Suppose $\GB(\mathcal Y)$ is a developable complex of groups with qi condition and $\YY_1$ is a connected subcomplex of $\YY$. We shall say that $\YY_1 $ is
{\em injective} if the inclusion morphism $\GB(\YY_1)\map \GB(\YY)$ induces an injective homomorphism at the level of fundamental groups.

Let $G_1$ be the image of $\pi_1(\GB(\YY_1))$ in $\pi_1(\GB(\YY))=G$, say. Then $\YY_1 $ is {\em cone-injective} if coning off the cosets of all the face subgroups of $G$ gives 
a qi embedding from the coned off Cayley graph of $G_1$ to the coned off Cayley graph of $G$ (or equivalently the $G_1$-equivariant inclusion 
$\BB_1\map  \BB$ as in Proposition \ref{prop1: cplx gps}(5) is a qi embedding).

We shall say that $\YY_1$ is {\em good} if it is both injective and 
cone-injective.

\end{defn}

We  have the following corollary. 
\begin{cor}\label{cor: fiber limit set}
Suppose $\GB(\YY)$ is a developable complex of infinite hyperbolic groups with qi condition over
a finite connected simplicial complex $\YY$ and suppose $G=\pi_1(\GB(\YY))$
is also hyperbolic. Then for all $y\in V(\YY)$, $g\in G$ we have $\Lambda_G(gG_yg^{-1})=\partial G$.
\end{cor}

\begin{proof}
Let $\pi:\XX\map \BB$
 be a metric graph bundle satisfying the properties of Proposition \ref{prop1: cplx gps}.
Let $K=gG_vg^{-1}$. Then for all $x\in G$, $K$ and $xKx^{-1}$ each fix a vertex of $\BB$, say $u,w$ respectively. Also $\FF_u, \FF_w$ are invariant under $K, xKx^{-1}$ respectively and these two induced actions are cocompact. 
By Lemma \ref{bundle lemma1} $Hd(\FF_u,\FF_w)<\infty$. Since the $G$-action on $\XX$ is proper and cocompact,
Lemma \ref{Hd lemma} then implies that $Hd(K, xKx^{-1})<\infty$. Clearly, $Hd(xK, xKx^{-1})\leq d(1,x)<\infty$.
Thus $Hd(K, xK)<\infty$ for all $x\in G$. Hence, $\Lambda_G(K)=\Lambda_G(xK)$ for all $x\in G$ by Lemma \ref{limit set lemma}(1).
Also, since $K$ is an infinite subgroup of $G$, $\Lambda_G(K)\neq \emptyset$ by Lemma \ref{limit set lemma}(0).
Finally, by Lemma \ref{ct elem lemma}(2) we have $\Lambda_G(K)=\partial G$ because $\Lambda_G(K)=\Lambda_G(xK)=x\Lambda_G(K)$
for all $x\in G$ whence $\Lambda_G(K)$ is a nonempty closed $G$-invariant subset of $\partial G$.
\end{proof}

\begin{defn}{\em (Trees of hyperbolic metric spaces \cite{BF})}
Suppose $\TT$ is a tree and $\XX$ is a metric space. Then a map $\pi:\XX\map \TT$ is called a tree of hyperbolic metric spaces
with qi embedded condition if there are constants $\delta\geq 0$, $K\geq 1$ and a function 
$f:\NN\map \NN$ with the following properties: 

(1) For all $v\in V(\TT)$, $\XX_v=\pi^{-1}(v)$ is a geodesic metric space with the induced path metric $d_v$, induced from
$\XX$. Moreover, with respect to these metrics the inclusion maps $\XX_v \map \XX$ are uniformly metrically proper
as measured by $f$.

(2) Suppose $e$ is an edge of $\TT$ joining $v,w\in V(\TT)$ and $m_e\in \TT$ is the midpoint of this edge.
Then $\XX_e=\pi^{-1}(m_e)$ is a geodesic metric space with respect to the induced path metric $d_e$ from $\XX$. 
Let $[v,w]$ denote the edge $e$ from $v$ to $w$. Then
moreover, there is a map $\phi_e:\XX_e\times [v,w]\map \pi^{-1}(e)\subset \XX$ such that\\ (i) $\pi\circ \phi_e$
is the projection map onto $[v,w]$. \\ (ii) $\phi_e$ restricted to $(v,w)\times \XX_e$ is an isometry onto $\pi^{-1}(\stackrel{\circ}{e})$ 
where $\stackrel{\circ}{e}$ is the interior of $e$.\\ (iii) $\phi_e$ restricted to $\XX_e\times \{v\}$ and
$\XX_e\times \{w\}$ are $K$-qi embeddings from $X_e$ into $\XX_v$ and $\XX_w$ respectively with respect to the
induced path metric $d_e$ on $\XX_e$, and $d_v, d_w$ on $\XX_v, \XX_w$ respectively.
\end{defn}

Given a tree of hyperbolic metric spaces with qi embedded condition it is convenient to replace the vertex and
edge spaces by quasi-isometric metric graphs and glue them using the maps $\phi_e$ to get a tree of hyperbolic metric graphs.
This is an example of a tree of metric graphs obtained by discretizing the classical Bass-Serre tree
of spaces (see \cite{scott-wall} for a topological exposition of Bass-Serre theory and \cite[Section 3]{ps-limset} for the discretized version).
The universal cover of a finite graph of spaces is a source of examples for
a tree of metric spaces \cite{scott-wall}. The notion of a graph of groups satisfying the qi embedded condition was introduced by Bestvina-Feighn in
\cite{BF}. Such a graph of groups gives rise to a tree of metric spaces satisfying the qi embedded condition, where the underlying tree is the corresponding Bass-Serre tree.

\subsection{Cannon-Thurston maps} In this subsection, we collect together various existence theorems for
Cannon-Thurston maps along with properties of CT laminations 
\begin{defn}\cite{CTpub,mitra-ct,mitra-survey}
Suppose $f:Y\map X$ is a map between hyperbolic metric spaces. We say that $f$ admits a Cannon-Thurston map (or a CT map for short)
if $f$ induces a continuous map $\partial f:\partial Y\map \partial X$. Equivalently, for all 
$\xi\in \partial Y$, there exists $\partial f(\xi) \in \partial X$ such that for
any sequence $\{y_n\}$ in $Y$ converging to $\xi$, $\{f(y_n)\}$ converges to $\partial f(\xi)$. Further,
$\partial f$ is required to be continuous.
\end{defn}
We refer to \cite{CTpub} for the origin of CT maps and to \cite{mahan-icm} for a survey.
Suppose $H<G$ are hyperbolic groups. Suppose $\Gamma_G$, $\Gamma_H$ are Cayley graphs of $G,H$ respectively
with respect to some finite generating sets. Since $G$ is identified with $V(\Gamma_G)$,
we have a natural map $i: H\map \Gamma_G$. This map can be extended to a coarsely well-defined map $\Gamma_H\map \Gamma_G$
by sending any point on an edge joining $h_1, h_2\in H$ to $i(h_1)$ or $i(h_2)$. If the map $\Gamma_H\map \Gamma_G$
admits a CT map then we  say that the inclusion map $H\map G$ admits a CT map. 

 The first and fourth parts of the following lemma are standard
and follow from the definitions of CT maps and limit sets. For the second part  see \cite[Chapter III.H]{bridson-haefliger}
for instance. The third part follows exactly as in \cite[Lemma 2.1]{mitra-pams}.

\begin{lemma}{(Properties of CT maps)}\label{ct elem lemma}
(1) Suppose $X,Y, Z$ are hyperbolic metric spaces and there
exist maps $g:Z\map Y, f: Y\map X$ admitting
CT maps $\partial g:\partial Z\map \partial Y$, $\partial f:\partial Y\map \partial X$. Then the composition
$f\circ g:Z\map X$ admits a CT map and $\partial (f\circ g)=\partial f\circ \partial g$.\\
(2) If $f:Y\map X$ is a qi embedding then there is an injective CT map $\partial f: \partial Y\map \partial X$. A
CT map induced by a quasi-isometry is a homeomorphism.
In particular the action of a hyperbolic group $G$ on its Cayley graph induces an action of $G$ on $\partial G$ by homeomorphisms.
This action is minimal, i.e. there is no proper nonempty closed subset of $\partial G$ invariant under $G$, provided $G$ is non-elementary.\\
(3) Suppose that 
a hyperbolic group $G$ acts by isometries on a hyperbolic metric space $X$ and the action
is properly discontinuous. Suppose $x\in X$ and that
  the orbit map $h:G\map X$ given by $h(g)= gx$ admits a CT map.
If $h$ is not a qi embedding then there are points $\xi_1\neq \xi_2\in \partial G$ such that
$\partial h(\xi_1)=\partial h(\xi_2)$. 
In particular this is true for a hyperbolic subgroup $H$ of a hyperbolic group $G$ if the inclusion
$H\map G$ admits a CT map.\\
(4) If $H<G$ are hyperbolic groups and the inclusion $i: H\map G$ admits a CT map $\partial i:\partial H\map \partial G$
then $\partial i(\partial H)=\Lambda_G(H)$. 
\end{lemma}

The following is an immediate consequence of the Milnor-{\v S}varc lemma and the
above Lemma.

\begin{cor}\label{cor: ct lemma}
Suppose $G$ is a hyperbolic group acting on a hyperbolic metric space $X$  properly and
cocompactly by isometries. Suppose that $H$ is a finitely generated subgroup of $G$ and that there is a subset $Y$ of $X$ invariant under
the $H$-action with the following properties:\\ (1) $Y$ is a hyperbolic metric space with respect to
the induced length metric from $X$.\\ (2) The inclusion $Y\map X$ admits a CT map. \\(3) The $H$-action
on $Y$ is proper and cocompact. \\ Then $H$ is hyperbolic and the inclusion $H\map G$ admits a CT map.
\end{cor}

The non-injectivity of CT maps motivates the following definition which will be crucial in this paper, cf.\ \cite{mitra-endlam, mahan-rafi}.
\begin{defn}
Suppose $f:Y\map X$ is a map between hyperbolic metric spaces which admits a CT map.
Then the {\em Cannon-Thurston lamination} ({\em or  CT lamination} for short) for this map is given by 
$$ \LL_{X}(Y)=\{ (\xi_1,\xi_2)\in \partial Y\times \partial Y: \xi_1\neq \xi_2, \, \partial f(\xi_1)=\partial f(\xi_2)\}.$$

If $\alpha$ is a (quasi)geodesic line in $Y$ such that $(\alpha(-\infty), \alpha(\infty))\in \LL_{X}(Y)$ then we say that
$\alpha$ is a {\em leaf} of the CT lamination $ \LL_{X}(Y)$.
\end{defn}

Let $f:Y\map X$ be a map between hyperbolic metric spaces  admitting a CT map. Let $Z \subset Y$ be a qi-embedded subset, so that $Z$ is hyperbolic and $\partial Z$ embeds in $\partial Y$ (Lemma \ref{limit set lemma}). 

\begin{defn}\label{defn-carry}
We say that a leaf $\alpha \subset Y$ of $\LL_{X}(Y)$ is carried by $Z$ if $\alpha$
lies in a bounded neighborhood of $Z$. Equivalently (by quasiconvexity of $Z$), $\alpha(\pm \infty)\in \partial Z (\subset \partial Y)$.
\end{defn}

A consequence of Proposition \ref{general thm} is the following.

\begin{cor}\label{general cor}
Suppose $G_1<G$ is a hyperbolic subgroup of a hyperbolic group such that the inclusion $G_1\map G$ admits a CT map.
Further let $H, K$ be hyperbolic subgroups of $G_1$ with the following properties:\\
\noindent
(i) The inclusion $K\map G_1$ admits a CT map.\\
(ii) $\Lambda_{G_1}(K)=\partial G_1$.\\
(iii) $H$ is a qi embedded subgroup of $G_1$ with $[G_1:H]=\infty$. \\
(iv) $H$ is not qi embedded in $G$.\\
Then a CT map for the pair $(H,G)$ exists and for any leaf $\alpha$ of the CT lamination $\LL_G(H)$ which is contained in a finite neighborhood of $K$
there exists a finitely generated subgroup $K_1$ of $H\cap K$, such that 
the following hold: \\
(1) $\alpha$ is contained in a finite neighborhood of $K_1$ (in a Cayley graph of $G$). \\
(2)  $\Lambda_K(K_1)\neq \partial K$. In particular, $[K:K_1]=\infty$.\\
Hence,\\
(3) $K_1$ is not qi embedded in $G$, i.e. $K_1$ is distorted in $G$; and\\
(4)  if $K_1$ is qi embedded in $K$ then $K_1$ supports a leaf of the CT lamination $\LL_G(K)$.
\end{cor}

\begin{proof}
 We first note that since the inclusion $G_1\map G$ admits a CT map and $H$ is qi embedded in $G_1$, the inclusion
$H\map G$ admits a CT map. Hence, by Lemma \ref{ct elem lemma}, $\LL_G(H)$ is non-empty. We now apply Proposition \ref{general thm}.
Since $G_1$ is hyperbolic it is finitely presented. See for instance \cite[Corollary 2.2A]{gromov-hypgps}.
Using Proposition \ref{general thm} with $G_1$ in place of $G$ and $\alpha$ in place of $A$ we have a finitely generated
subgroup $K_1$ of $H\cap K$ such that conclusion (1) of the Corollary holds.

To prove (2), suppose $\Lambda_K(K_1)= \partial K$.  Since the inclusion $i: K\map G_1$ admits a CT map
$\partial i: \partial K\map \partial G_1$, it follows that $\Lambda_{G_1}(K_1)=\partial i(\Lambda_K( K_1))=\partial i(\partial K)$.
However, by Lemma \ref{ct elem lemma}(4) $\Lambda_{G_1}(K)=\partial i(\partial K)=\partial G_1$. Hence,
$\Lambda_{G_1}(K_1)=\partial G_1$. On the other hand  $\Lambda_{G_1}(K_1)\subset \Lambda_{G_1}(H)$. Hence,
$\Lambda_{G_1}(H)=\partial G_1$,  contradicting hypothesis (iii).
Hence,
$\Lambda_K(K_1) \neq \partial K$, forcing
 $[K:K_1]=\infty$ and proving (2).

To prove (3), we again argue by contradiction.
Suppose $K_1$ is qi embedded in $G$. Then $\LL_G(K_1)$
is empty by Lemma \ref{ct elem lemma}. Further,  $K_1$ is qi embedded in $H$ as well. Thus $\alpha$ is a leaf of
the CT lamination $\LL_G(H)$ supported by $K_1$, forcing 
$\LL_G(K_1)$ to be non-empty. This   contradiction  proves (3).

The inclusions $K\map G_1$ and $G_1\map G$ admit CT maps. Hence, the inclusion $K\map G$ admits a CT map by Lemma \ref{ct elem lemma}(1).
Now suppose $K_1$ is qi embedded in $K$. Since $\alpha$ is contained in a finite neighborhood of $K_1$,
there is a (bi-infinite) 
quasigeodesic line $\beta$ of $K$ contained in  $K_1$ with $Hd(\alpha, \beta)<\infty$. To see this, suppose $\alpha\subset N_D(K_1)$
where the neighborhood is taken in a Cayley graph of $G$. Then for all $t\in [0,\infty)$ one may choose
$\beta(t)$ to be a point of $K_1$ such that $d(\alpha(t), \beta(t))\leq D+1$. Thus, $\beta$ is a 
quasigeodesic in (a Cayley graph of) $K$  by Lemma \ref{lem: prop emb}. Therefore, $\beta$ is  a leaf of the CT lamination $\LL_G(K)$ 
carried by $K_1$. This finishes the proof of (4). 
\end{proof}

We now recall some existence theorems for CT maps:

\begin{theorem}{\em ( \cite[Theorem 3.10, Corollary 3.11]{mitra-trees})}
Suppose $\pi:\XX\map \TT$ is a tree of uniformly hyperbolic metric spaces with qi embedded condition.
If the total space $\XX$ is hyperbolic then for all $v\in V(\TT)$ the inclusion map $\XX_v\map \XX$
admits a CT map.

In particular given a finite graph of hyperbolic groups $(\mathcal G, \mathcal Y)$ with qi embedded condition, 
if $G=\pi_1(\mathcal G, \mathcal Y)$ is hyperbolic then for all $v\in V(\YY)$, the inclusion map $G_v\map G$ 
admits a CT map.  
\end{theorem}

\begin{theorem}{\em(\cite{mitra-ct})}\label{kernel ct}
Suppose we have a short exact sequence of hyperbolic groups 
$$1\rightarrow K\stackrel{i}{\rightarrow} G\stackrel{\pi}{\rightarrow} H\rightarrow 1.$$
Then a CT map exists for the inclusion $K\map G$.
\end{theorem}

\begin{theorem}{\em (\cite[Theorem 5.3]{mahan-sardar})}\label{ct for fiber of bundles}
Suppose $\pi:\XX\map \BB$ is a metric graph bundle where $\XX$,  and all the fibers are uniformly
hyperbolic. Also we assume that the fibers are non-elementary (i.e.\ the barycenter maps are uniformly coarsely surjective).
Then for any fiber $\FF_b, b\in V(\BB)$, the inclusion $\FF_b\map \XX$ admits a CT map.
\end{theorem}

We note that  hyperbolicity of $\BB$ is not an assumption for Theorem \ref{ct for fiber of bundles}
since it follows from the  hypotheses (\cite{mosher-hypextns} and \cite[Proposition 2.10]{mahan-sardar}).
We have the following corollary. 
\begin{cor}\label{cor: fiber ct}
Suppose $\GB(\YY)$ is a developable complex of non-elementary hyperbolic groups with qi condition over
a finite connected simplicial complex $\YY$ and suppose $G=\pi_1(\GB(\YY))$
is also hyperbolic. Then for all $y\in V(\YY)$, $g\in G$ the inclusion $gG_yg^{-1}\map G$ admits a CT map.
\end{cor}
\begin{proof} 
By Proposition \ref{prop1: cplx gps} there is a metric graph bundle $\pi:\XX\map \BB$ along with simplicial actions
of $G$ on $\XX$ and $\BB$ such that $\pi$ is $G$-equivariant.
Further, all the fibers are uniformly hyperbolic
and $\XX$ is hyperbolic. Also, there is an isomorphism of graphs $p:\BB/G\map \YY^{(1)}$ such that
for all $y\in \YY^{(0)}$, $\{G_v: v\in p^{-1}(y)\}$, is the set of all conjugates of $G_y$. Moreover, each subgroup
$G_v$ acts properly and cocompactly on $\FF_v$. By Theorem \ref{ct for fiber of bundles} inclusion
of each fiber in $\XX$ admits a CT map. Therefore, we are done by Corollary \ref{cor: ct lemma}.
\end{proof}

The next two theorems \cite{kap-ps,mbdl2} extend the above theorems. For trees of metric
spaces, we have  the following.

\begin{theorem}{\em (CT maps for subtrees of spaces \cite[Chapter 8, Section 8.6]{kap-ps})}\label{ct for subtree}
Suppose $\pi: \XX\map \TT$ is a tree of hyperbolic metric spaces with qi embedded condition and such that
the total space $\XX$ is hyperbolic. Suppose $\TT_1\subset T$ is a subtree and $\XX_1=\pi^{-1}(\TT_1)$.
Then $\XX_1$ is hyperbolic and the inclusion $\XX_1\map \XX$ admits a
CT map.
\end{theorem}

One then immediately obtains: 
\begin{cor}{\em (CT maps for subgraph of groups, \cite[Chapter 8, Section 8.11]{kap-ps})}\label{ct for subgraph}
Given a finite graph of hyperbolic groups $(\mathcal G, \mathcal Y)$ with qi embedded condition, 
and a connected subgraph $\mathcal Y_1\subset \YY$, if $G=\pi_1(\mathcal G, \mathcal Y)$ is hyperbolic
then so is $G_1=\pi_1(\mathcal G, \mathcal Y_1)$ and the inclusion map $G_1\map G$ admits a CT map.
\end{cor}

For a metric graph bundle, we have:
\begin{theorem}{\em (CT maps for restriction bundles \cite[Section 6]{mbdl2})}\label{thm-subbdlhyp}
Suppose $\pi:\XX\map \BB$ is a metric graph bundle where $\XX$, $\BB$ and all the fibers are uniformly
hyperbolic. Also we assume that the fibers are non-elementary. Suppose $\BB_1\subset \BB$ is a connected
subgraph such that the inclusion map $\BB_1\map \BB$ is a qi embedding. 
Let $\XX_1=\pi^{-1}(\BB_1)$. Then $\XX_1$ is hyperbolic and the inclusion map
$\XX_1\map \XX$ admits a CT map.
\end{theorem}

We note that the conclusion about hyperbolicity of $\XX_1$ in the above theorem is a consequence of
\cite[Remark 4.4]{mahan-sardar}. The theorem immediately implies the following. 
\begin{cor}\label{ct for pullback}(\cite[Section 6]{mbdl2})
(1) Suppose we have a developable complex of non-elementary hyperbolic groups $\GB(\YY)$ with qi condition
over a finite connected simplicial complex $\YY$, such that $G=\pi_1(\GB(\YY))$ is hyperbolic. 
Suppose $\YY_1$ is a good subcomplex of $\YY$ (as per Definition \ref{defn: good subcomplex}) and $G_1=\pi_1(\GB(\YY_1))$.
Then $G_1$ is hyperbolic and the inclusion $G_1\map G$ admits a CT map.\\
(2) Given an exact sequence of infinite hyperbolic groups 
$$1\rightarrow K\stackrel{i}{\rightarrow} G\stackrel{\pi}{\rightarrow} Q\rightarrow 1,$$
if $Q_1<Q$ is qi embedded then $G_1=\pi^{-1}(Q_1)$ is hyperbolic and the inclusion $G_1\map G$ admits a CT map.
\end{cor}

We  also have the following  statements for the leaves of the CT lamination.
The CT lamination for a subtree of spaces  satisfies the following.

\begin{theorem}{\em (CT lamination for subtree of spaces, \cite[Chapter 8, Section 8.7]{kap-ps})}\label{subtree lamination}
Assume the hypotheses of Theorem \ref{ct for subtree}. Suppose that $\alpha$ is a leaf of
$\LL_{\XX}(\XX_1)$. There exist $v\in V(\TT_1)$ and a geodesic line $\beta\subset \XX_v$ such that
$\beta$ is also a quasi-geodesic in $\XX_1$ and $Hd(\alpha, \beta)<\infty$.
Moreover,  there is a point $\xi\in \partial \TT\setminus \partial \TT_1$ and a quasi-isometric
lift  $\gamma$ of the geodesic in $T$ joining $v$ to $\xi$ such that $\lim_{n\map \infty} \alpha(n)=\gamma(\infty)$.
\end{theorem}

One then immediately obtains: 
\begin{cor}\label{lamination for subtree}{\em (CT lamination for subgraph of groups, \cite[Chapter 8, section 8.11]{kap-ps})}\label{lamination for subgraph}
Assume the hypotheses of Corollary \ref{ct for subgraph}.
Given a leaf $\alpha$ of the CT lamination $\LL_G(G_1)$, there exist $v\in V(\YY_1)$ and $g\in G_1$
such that $\alpha$ is contained in a finite neighborhood of $gG_vg^{-1}$. More precisely
there is a geodesic $\beta$ in $gG_vg^{-1}$ which is also a quasi-geodesic in $G_1$ and $Hd(\alpha, \beta)<\infty$.
\end{cor}
Similar statements hold also for the restriction bundle of a metric graph bundle:

\begin{theorem}{\em (CT lamination for restriction bundle, \cite[Section 6]{mbdl2})}\label{pullback lamination}
Assume the hypotheses of Theorem \ref{thm-subbdlhyp}.
Suppose that $\alpha$ is a leaf of
$\LL_{\XX}(\XX_1)$. Then for all $v\in V(\BB_1)$ there is a geodesic line $\beta\subset \XX_v$ such that
$\beta$ is also a quasi-geodesic in $\XX_1$ and $Hd(\alpha, \beta)<\infty$.
Moreover, there exists $\xi\in \partial \BB\setminus \partial \BB_1$ and a qi section  $\gamma$ over a geodesic
in $\BB$ converging to $\xi$ such that $\lim_{n\map \infty} \alpha(n)=\gamma(\infty)$.
\end{theorem}

This immediately gives: 
\begin{cor}\label{lamination for pullback}{\em (CT lamination for subcomplexes of groups, \cite[Section 6]{mbdl2})}

(1) Assume the hypotheses of Corollary \ref{ct for pullback}(1). Then for any $v\in V(\YY_1)$ and $g\in G$,
any leaf $\alpha$ of the CT lamination $\LL_{G}(G_1)$ is contained in a finite neighborhood of  $gG_vg^{-1}$.

(2) Assume the hypotheses  of Corollary \ref{ct for pullback}(2).
Then given any leaf $\alpha$ of the CT lamination $\LL_G(G_1)$ there is a geodesic line $\beta$ in $K$ 
such that it is a quasi-geodesic of $G_1$ and $Hd(\alpha, \beta)<\infty$.
\end{cor}

\section{Distortion comes from  fibers}\label{sec-propagate}
In this section, we prove the main theorems of this paper
and provide precise statements in support of Scholium \ref{schol-propagate}.
 The first two subsections correspond to the two main sources of examples of metric graph bundles
that we mentioned in section 3.1. In the last section we discuss applications to graphs of groups.

\subsection{Exact sequences of hyperbolic groups}\label{sec-es}
Let $$\,1\map K\map G\stackrel{\pi}{\map} Q\map 1 \quad (*)\,$$ be an exact sequence of infinite hyperbolic groups
and $Q_1< Q$ be a qi embedded subgroup. Let $G_1=\pi^{-1}(Q_1)$. Then $G_1$ is a hyperbolic group by \cite[Remark 4.4]{mahan-sardar}.
Suppose $H$ is a qi embedded, infinite index subgroup of $G_1$. The following theorem asserts  that if $H$ is distorted  in $G$ (cf.\ Definition \ref{defn-disto}; \cite[Chapter 4]{gromov-ai}), then
this is due to the presence of a finitely generated subgroup $K_1
< K \cap H$ such that $K_1$ is distorted in $G$. In other words, the cause
of distortion lies in finitely generated subgroups of the fiber group $K$.

\begin{theorem}\label{general fiber} With the above assumptions and notation,
	suppose $H$ is not qi embedded in $G$. Then there is a finitely generated, infinite subgroup $K_1$ of
	$K\cap H$ such that $\Lambda_{K}(K_1) \neq \partial K$,
	 $[K:K_1]=\infty$ and $K_1$  is distorted in $G$.
\end{theorem}

\begin{proof}
The inclusion $G_1\map G$ admits a CT map by  Corollary \ref{ct for pullback}(2). 
Since $H$ is qi embedded in $G_1$, there is a CT map for the inclusion $H\map G_1$ by Lemma \ref{ct elem lemma}(2).
 By Lemma \ref{ct elem lemma}(1) the inclusion $i_{H,G}:H\map G$ also admits a CT map 
$\partial i_{H,G}: \partial H\map \partial G$. Since $H$ is not qi embedded in $G$, by Lemma \ref{ct elem lemma}(3)
there is a geodesic line $\gamma$ in $H$ whose end points are identified by the CT map 
$\partial i_{H,G}$. By Corollary \ref{lamination for pullback}(2), $\gamma$ is contained in a bounded neighborhood of $K$.

Recall that $\Lambda_G(K)=\partial G$ by Lemma \ref{limit set lemma}(2)(i). Also the inclusion $K\map G_1$
admits a CT map by Theorem \ref{kernel ct}.  The theorem now follows immediately from Corollary \ref{general cor}.
\end{proof}

\noindent {\bf Surface group fibers:} We specialize the exact sequence $(*)$ above to the case where 
\begin{enumerate}
\item $K= \pi_1(\Sigma)$, 
with $\Sigma$ a closed orientable surface of genus at least $2$. 
\item $Q$ is a subgroup of $MCG(\Sigma)$.
\end{enumerate}
 It follows
\cite{hamen,kl-coco,mahan-sardar} that $Q$ is an infinite convex cocompact subgroup
of the mapping class group $MCG(\Sigma)$.
We shall need the following Theorem by Dowdall, Kent and Leininger \cite[Theorem 1.3]{dkl} (see also \cite[Theorem 4.7]{mahan-rafi} for a different proof).
\begin{theorem} \label{thm-dklmr1}
	Consider the exact sequence $(*)$ with $K=\pi_1(\Sigma)$ ($\Sigma$ closed) and $Q$ convex cocompact. 
	Let $K_1$ be a finitely generated infinite index subgroup of $K$. Then  $K_1$ is quasi-convex  in $G$.
\end{theorem}

 Suppose $Q_1< Q$ is a qi embedded subgroup and $G_1=\pi^{-1}(Q_1)$. Then $Q_1$ is also convex cocompact \cite{hamen,kl-coco}. Further, by  Corollary 
 \ref{ct for pullback}(2),
$G_1$ is also hyperbolic. We continue using the notation before Theorem \ref{general fiber}.

\begin{theorem}\label{surface fiber} 
	Suppose $H<G_1$ is an infinite index, quasiconvex subgroup. Then $H$ is quasiconvex  in $G$.
\end{theorem}

\begin{proof}
Suppose not. Then, by Theorem \ref{general fiber}, there exists a finitely generated, infinite subgroup $K_1$ of
$K\cap H$ such that 
\begin{enumerate}
\item $[K:K_1]=\infty$.
\item $K_1$ is not quasiconvex in $G$.
\end{enumerate}
This contradicts Theorem \ref{thm-dklmr1}.
\end{proof}


\noindent {\bf Free group fibers:} Next, instead of specializing to $K=\pi_1(\Sigma)$, we specialize the exact sequence $(*)$  to the case where 
\begin{enumerate}
	\item $K= F_n$, the free group on $n$ generators
	($n>2$), 
	\item $Q$ is a subgroup of $Out(F_n)$.
\end{enumerate}
and continue with the rest of the
notation before Theorem \ref{general fiber}. We recall from \cite{dt1, hamhen} that a subgroup
$Q< Out(F_n)$  is said to be {\em convex cocompact}    if some (and hence any)
orbit of $Q$ in the free factor complex $\FF_n$ is qi embedded. We say that 
$\phi \in Out(F_n)$ is hyperbolic, if $F_n \rtimes_\phi \Z$ is hyperbolic.
Also, we say that
$Q< Out(F_n)$ is {\em  purely atoroidal} if every element of $Q$ is hyperbolic. We shall need the following:

\begin{theorem} \label{thm-free}   \cite[Theorem 7.9]{dt1}\cite[Theorem 5.14]{mahan-rafi}
Consider the exact sequence $(*)$ with $K=F_n$, $(n>1)$ and $Q$ 
purely atoroidal and convex cocompact in $Out(F_n)$. 
Let $K_1$ be a finitely generated infinite index subgroup of $K$. Then  $K_1$ is quasi-convex  in $G$.
\end{theorem}

Next, suppose $Q_1< Q$ is a qi embedded subgroup and $G_1=\pi^{-1}(Q_1)$. Then $Q_1$ is purely atoroidal convex cocompact \cite{dt1,dowtay,hamhen}.  By   Corollary 
\ref{ct for pullback}(2),
$G_1$ is also hyperbolic. We continue with the notation before Theorem \ref{general fiber}.

\begin{theorem}\label{freefiber} 
Under the hypotheses of Theorem \ref{thm-free}, let $Q_1< Q$ be qi embedded and $G_1=\pi^{-1}(Q_1)$.
	Suppose $H<G_1$ is an infinite index, qi embedded subgroup. Then $H$ is qi embedded in $G$.
\end{theorem}

\begin{proof}
The proof is an exact replica of the proof of
Theorem \ref{surface fiber} with the use of Theorem \ref{thm-dklmr1} replaced by Theorem \ref{thm-free}.
\end{proof}

\noindent {\bf A counter-example:} We give an example to show that the conclusion of Theorem \ref{freefiber} fails if $Q$ is not assumed to be convex cocompact even if $G$ is hyperbolic. 
It follows from Theorem \ref{general fiber} that
 the failure is due to the
existence of an infinite index quasiconvex subgroup $K_1$ of the normal free subgroup $K$
such that $K_1$ is distorted in $G$. 
In fact, the examples
below  show that the corresponding
results of \cite{dt1,mahan-rafi} fail without the convex cocompactness assumption. We start with the following construction due to Uyanik
\cite[Corollary 1.5]{uyanik} and Ghosh \cite[Theorem 5.6]{ghosh2}.

\begin{defn}
Two automorphisms $\phi, \psi \in Out(F_n)$ are said to be \emph{commensurable}
if there  exist $m, l \neq 0$ such that $  \phi^m = \psi^l$.
\end{defn}

\begin{theorem}\label{thm-uyagh}\cite{uyanik,ghosh2}
Let $\phi \in Out(F_n)$ be a fully irreducible and atoroidal outer automorphism.   Then,  for any (not necessarily fully irreducible) atoroidal outer automorphism $\psi \in Out(F_n)$ which is not commensurable with $\phi$, there exists $N \in \mathbb{N}$ such that, for all $m, l \geq N$, the subgroup $Q=\langle \phi^m, \psi^l \rangle <
Out(F_n)$ is purely atoroidal and the corresponding extension $G$ of $F_n$ by $Q$
is hyperbolic.
\end{theorem}

Now, let $F_n = A\ast B$ be the product of two free
factors, each of rank 3 or more. Let $\psi_A$ and $\psi_B$ be fully irreducible purely atoroidal automorphisms of $A, B$ respectively, and let
$\psi = \psi_A \ast \psi_B$. Let $\phi$ be a fully irreducible purely atoroidal automorphism of $F_n$. Theorem \ref{thm-uyagh} furnishes a 
family of purely atoroidal subgroups $Q=\langle \phi^m, \psi^l \rangle <
Out(F_n)$ for all $m, l$ large enough such that 
the corresponding extension $G$ of $F_n$ by $Q$
is hyperbolic. Now, let $Q_1 =\langle \phi^m \rangle < Q$
and $G_1 < G$ be given by $\pi^{-1} (Q_1)$, where $\pi:G \to Q$ is the natural quotient map. Similarly, let $Q_2 =\langle \psi^l \rangle < Q$
and $G_2 < G$ be given by $\pi^{-1} (Q_2)$.
Then $A<F_n<G_1$ is quasiconvex by Theorem \ref{thm-free}. But $A<F_n<G_2$ is not quasiconvex, since $A$ is distorted in $ (A \rtimes_{\psi_A^l} \Z) < G_2$.
Hence $A$ is not quasiconvex in $G$. 

In fact $A$ (or $B$) can be used 
for building more complicated examples as well. Pick any element
$\gamma \in G_1 \setminus A$. Then by the usual ping-pong argument
$A_2 = \langle A, \gamma^m\rangle $ is quasiconvex in $G_1$ for all large enough $m$. But $A_2$ is not quasiconvex in $G$ since it contains the distorted subgroup
$A$. 
Thus,
the example above can be refined to prove the following. 
\begin{prop}\label{prop-ctreg}
Let $F_n, \phi, G_1$ be as 
above. Let $H < G_1$ be an infinite quasiconvex subgroup of infinite index such that $H$ is not virtually cyclic.
Then there exists a finite index subgroup $G_1' < G_1$, so that 
\begin{enumerate}
\item $G_1' = F_m \rtimes \Z$, where the semi-direct product is given by a   positive power $\phi_1$
of a lift  of $\phi$ to a finite index subgroup $F_m$ of $F_n$,
 \item a purely atoroidal  automorphism $\psi$ of $F_m$ such that 
 $Q=\langle \psi, \phi_1 \rangle < Out(F_m)$ is free, and $G = F_m \rtimes Q$ is hyperbolic.
 \item $G_1' \cap H$ is of finite index in $H$ and is not quasiconvex in $G$.
\end{enumerate}
\end{prop}

\begin{proof}	
Since $H < G_1$ is an infinite quasiconvex subgroup of infinite index  that  is not virtually cyclic, $H \cap F_n$ is infinite, free of rank greater than one.
Let $K_1 < H \cap F_n$ be a free subgroup of finite rank greater than 2.
Note that $K_1$ is quasiconvex in $G_1$ by Theorem \ref{freefiber}.
 Now, by Hall's theorem,
there exists a finite index subgroup  $F_k$  of $F_n$, containing $K_1$, such that $K_1$
is a free factor of $F_k$ (this is essentially the LERF property for free
groups \cite{scott-lerf}). By passing to a further finite index subgroup $F_m$ of $F_n$, we can assume that
\begin{enumerate}
\item $\phi$ lifts to an automorphism $\phi'$ of $F_m$.
\item   $K_1\cap F_m= K_2$ is a free factor of $F_m$.
\end{enumerate}
Let $F_m = K_2 * B$ be a free factor decomposition. Let $\psi_A, \psi_B$
be fully irreducible purely hyperbolic automorphisms of $K_2$, $B$ respectively as in the construction of the counterexample preceding
Proposition \ref{prop-ctreg}. (Note that we can assume, without loss of generality that $B$ has rank greater than 2 by passing to finite index subgroups if necessary.) Let $\psi_0 = \psi_A * \psi_B$ as before.

By work of Reynolds \cite{preynolds-gd} summarized in \cite[Theorem 4.10]{mahan-rafi}, a lift of a fully irreducible automorphism is 
fully irreducible.
Then Theorem \ref{thm-uyagh} furnishes a 
positive integer $N$ such that $Q=\langle (\phi')^l, \psi_0^m\rangle$ is 
a purely atoroidal free subgroup of $Out(F_m)$, and $G=F_m \rtimes Q$ is hyperbolic. Choosing $\psi = \psi_0^m$, $\phi_1 = (\phi')^l$
and $G_1' = F_m \rtimes \langle \phi_1\rangle$, it follows that $K_2$ is distorted in
$G$ as in the construction of the counterexample described before
the present proposition. Since $H$ contains $K_2$, and $K_2$ 
 (being a finite index subgroup of $K_1$) is quasiconvex in $G_1'$,
it follows that
$G_1' \cap H$ is not quasiconvex in $G$. We note in conclusion that since $[G_1:G_1']< \infty$, therefore, $[H:G_1' \cap H]< \infty$.
\end{proof}


\subsection{Complexes of hyperbolic groups with qi condition} 
Suppose $\YY$ is a finite connected simiplicial complex and $\GB(\YY)$ is a developable
complex of nonelementary hyperbolic groups with qi condition over $\YY$. Moreover, suppose $G=\pi_1(\GB(\YY))$
is hyperbolic. Let $\YY_1$ be a good subcomplex (as per Definition \ref{defn: good subcomplex}) 
of $\YY$, and  $G_1=\pi_1(\GB(\YY_1))$. Then by Corollary \ref{ct for pullback} $G_1$ is
hyperbolic and the inclusion $G_1\map G$ admits a CT map. \\

\noindent {\bf Similarities and differences with Theorem \ref{general fiber}:}\\
We briefly outline the similarities and differences
between Theorems \ref{thm-cx} and Theorem \ref{thm-graph} below on the one hand, and Theorem \ref{general fiber} and its consequences on the other. Formally, all three -- Theorem \ref{general fiber}
 Theorem \ref{thm-cx} (3) and Theorem \ref{thm-graph} (2) -- identify a distorted
finitely generated subgroup $K_1 \subset H \cap K$, where $H$ and $K$ have different interpretations according to context. In all three cases, $H$ is a qi embedded subgroup of $G_1$, where
\begin{enumerate}
\item $G_1 = \pi^{-1}(Q_1) $ in Theorem \ref{general fiber}, and maybe thought of geometrically as a subbundle of $G$.
\item $G_1 = \pi_1(\mathbb{G} (\YY_1))$ in Theorem \ref{thm-cx}, where $\YY_1$ is a good subcomplex of a finite complex of groups.
\item $G_1$ is the fundamental group of a \emph{connected subgraph} of a finite graph 
of groups in Theorem \ref{thm-graph}.
\end{enumerate}

Again, $K$ is
\begin{enumerate}
	\item the normal subgroup in Theorem \ref{general fiber} and may be geometrically regarded as the fiber over a point.
	\item a face group in Theorem \ref{thm-cx}, and may  be geometrically regarded as the fiber over  (the barycenter of) a face.
	\item a vertex group in  Theorem \ref{thm-graph} and may again be geometrically regarded as the fiber over a point.
\end{enumerate}

The difference between Theorems \ref{thm-cx} and Theorem \ref{thm-graph} below on the one hand, and Theorem \ref{general fiber} on the other, lie in the applications. We do not have the analogs of Theorems \ref{surface fiber} of \ref{freefiber}. The crucial missing ingredient is a robust theory of laminations in the generality 
mentioned below:

\begin{prob}\label{qn-el}
Develop a theory of  laminations along the lines of \cite{mitra-endlam,kl10, kl15,dkt,mahan-rafi}
for
\begin{enumerate}
\item Hyperbolic complexes of groups,
\item Hyperbolic trees of spaces.
\end{enumerate}
\end{prob}

Thus, the main implication of the following
theorem is that an infinite index qi embedded subgroup $H$ of $G_1$ fails to be qi embedded in $G$ only if there
is  a face group, say $G_v$, $v\in V(\YY)$, 
and a finitely generated subgroup of
  $G_v \cap H$ that is
 not qi embedded in $G$.

\begin{theorem}\label{thm-cx}
Suppose $H$ is an infinite index, qi embedded subgroup of $G_1$ which is distorted in $G$. 
Then the inclusion $H\map G$ admits a CT map. Further, for 
any $g\in G_1$ and any $v\in V(\YY_1)$,
 and any leaf $\alpha$ of the CT lamination $\LL_G(H)$, $\alpha$
is contained in a finite neighborhood of $gG_vg^{-1}$.

Moreover, given a leaf $\alpha$ of $\LL_G(H)$, $g\in G_1$ and $v\in V(\YY)$
there is a finitely generated subgroup $K_1<H\cap gG_vg^{-1}$ such that the following hold:\\
(1) $\alpha$ is contained in a finite neighborhood of $K_1$.\\
(2) $\Lambda_{gG_vg^{-1}}(K_1)\neq \partial (gG_vg^{-1})$. In particular, $[gG_vg^{-1}:K_1]=\infty$.\\
(3) $K_1$ is a distorted subgroup of $G$. \\ 
(4) If $K_1$ is qi embedded in $gG_vg^{-1}$ then it supports a leaf of the CT lamination $\LL_G(gG_vg^{-1})$.
\end{theorem}

\begin{proof}
 By Corollary \ref{ct for pullback}(1) the inclusion $G_1\map G$ admits a CT map.
By Lemma \ref{ct elem lemma}(1) the inclusion $H\map G_1$ admits a CT map since $H$ is a qi embedded subgroup of $G_1$.
Hence, the inclusion $H\map G$ admits a CT map too. Also by lemma \ref{ct elem lemma}(4) $\mathcal L_G(H)\neq \emptyset$
since $H$ is distorted in $G$. Next, we note that $\alpha$ is a quasigeodesic in $G_1$ too since $H$ is qi embedded in $G_1$.
Thus it is a leaf of the CT lamination $\LL_G(G_1)$. Hence, by Corollary \ref{lamination for pullback}(1) it follows that 
for all $g\in G_1$ and $v\in \YY$,
$\alpha$ is contained in a finite neighborhood of $gG_vg^{-1}$.

The remaining parts of the theorem follow immediately from Corollary \ref{general cor}. We quickly check
the various hypotheses of  Corollary \ref{general cor}. We 
have already noted that the inclusion $G_1\map G$
admits a CT map. Let $K=gG_vg^{-1}$. Clearly, $H, K$ are hyperbolic groups.  By Corollary \ref{cor: fiber ct} the
inclusion $K\map G_1$ admits a CT map. By Corollary \ref{cor: fiber limit set} $\Lambda_{G_1}(K)=\partial G_1$. Lastly
$H$ is given to be qi embedded in $G_1$ with $[G_1:H]=\infty$. 
Thus, the  hypotheses of  Corollary \ref{general cor} are satisfied,
completing the proof of this theorem.
\end{proof}

To conclude this subsection, we mention a construction due to
Min \cite[Theorem 1.1]{min}  of a graph of groups where
\begin{enumerate}
\item Each edge and vertex group is a closed hyperbolic surface 
group.
\item The inclusion map of an edge group
into a vertex group takes the edge group injectively onto a subgroup
 of finite index in the vertex group.
 \item The resulting graph of groups is hyperbolic.
\end{enumerate}

By choosing the maps in the above construction carefully,
Min furnishes examples \cite[Section 5]{min} that are
 not abstractly commensurable to a surface-by-free group. In particular, the resulting groups are not commensurable to those occurring in the
 exact sequence $(*)$ in Section \ref{sec-es}. However, since the graphs of hyperbolic groups
in this construction satisfy  property (2), i.e. the qi condition, they give rise to metric graph bundles. Min's examples from \cite{min} were generalized 
to the context of free groups in \cite{ghosh-mj}.
In the next subsection we consider
more general graphs of hyperbolic groups. 

\subsection{Graphs of groups}
 Suppose  $(\mathcal G, \mathcal Y)$ is a  graph of hyperbolic groups where $\YY$ is a finite connected graph. Suppose $(\GG, \YY_1)$ is a subgraph of groups
where $\YY_1$ is a connected subgraph of $\YY$. Let $G=\pi_1(\mathcal G, \mathcal Y)$ and $G_1=\pi_1(\mathcal G, \mathcal Y_1)$.
Suppose $G$ is hyperbolic. Then by Corollary \ref{ct for subgraph} the group $G_1$ is also hyperbolic and the inclusion $G_1\map G$ admits a CT map.
As has been the theme of this paper, the theorem below shows that the existence of an infinite index, qi embedded subgroup $H$ of $G_1$ 
which fails to be qi embedded in $G$ implies 
that there is a vertex $v\in \YY_1$ and a finitely generated subgroup of $G_v\cap H$ which is not qi embedded in $G$.

\begin{theorem}\label{thm-graph}
Suppose $H$ is a qi-embedded,
infinite index subgroup of $G_1$ such that $H$ is not qi embedded in $G$. 
Then there exists a CT map for the pair $(H,G)$. Let
$\alpha$ be a leaf of the CT lamination for the inclusion $H\map G$.
Then
there exists $y\in V(\YY_1)$, $g\in G_1$ and a finitely generated subgroup $K_1<gG_yg^{-1}\cap H$ such that the
following hold:\\
(1) $\alpha$ is contained in a finite neighborhood of $K_1$.\\
(2) $K_1$ is distorted in $G$.\\
(3) If $K_1$ is qi embedded in $gG_yg^{-1}$ then it supports a leaf of the CT lamination $\LL_G(gG_yg^{-1})$.
\end{theorem}

We note that in this case 
the homomorphisms from the respective edge groups into the vertex groups are only qi-embeddings, and not necessarily quasi-isometries.
Thus the corresponding tree of metric spaces is not a metric bundle in general. Consequently Theorem \ref{thm-graph} above is not
a consequence of any of the results proved so far.

{\em Proof of Theorem \ref{thm-graph}:}
 The proof runs along the line of the proof of Corollary \ref{general cor} using Proposition \ref{general thm}.
By Corollary \ref{ct for subgraph} the group $G_1$ is hyperbolic and the inclusion $G_1\map G$ admits a CT map.
Also since $H$ is qi embedded in $G$ the inclusion $H\map G_1$ admits a CT map. Hence, by Lemma \ref{ct elem lemma}(1)
the inclusion $H\map G$ admits a CT map. Since $H$ is not qi-embedded in $G$ (by hypothesis),
 $\LL_G(H)$ is non-empty by Lemma \ref{ct elem lemma}. 
 
 Since $H$ is qi embedded in $G_1$, $\alpha$
is a quasigeodesic in $G_1$ and therefore it is a leaf of the CT lamination $\LL_G(G_1)$. Hence, by Corollary 
\ref{lamination for subgraph} there is $y\in V(\YY_1)$, $g\in G_1$, and a geodesic line $\beta$ in
$gG_yg^{-1}$, such that $Hd(\beta, \alpha)<\infty$ and $\beta$ is also a quasigeodesic in $G_1$.

We can now apply Proposition \ref{general thm} with $A=\alpha$. The triple of groups $(G_1, H, gG_yg^{-1})$
plays the role of $(G, H, K)$ in Proposition \ref{general thm}. It follows that there is a finitely generated infinite group $K_1<gG_yg^{-1}\cap H$
such that $\alpha$ is contained in a finite neighborhood of $K_1$. This verifies (1).

Conclusion (2), (3) follow
from (1) as in the proof of Corollary \ref{general cor}. We include the argument for the sake of 
completeness. We argue by contradiction by assuming $K_1$ is qi embedded in $G$. Then $K_1$ is qi embedded in $H$ as well. Hence $\alpha$ is
 a leaf of $\LL_G(H)$ supported by $K_1$. Hence $K_1$ is distorted in $G$,
 contradicting our assumption. 
 
Finally,  if $K_1$ is qi embedded in $K$ then it is hyperbolic.  Then $K_1$ supports $\beta$, which is a leaf
of the CT lamination $\LL_G(gG_yg^{-1})$. This completes the proof of the theorem. \qed  \\

\noindent {\bf Acknowledgments:} We are grateful to the anonymous referee for a careful reading of the manuscript and several helpful suggestions for improvement.

\end{document}